%% file: BEA_MultiStep_Matrix.tex
\documentclass{amsart}
\usepackage[foot]{amsaddr}
\usepackage{mathtools}
\usepackage{accents}

\usepackage{scalerel}
\usepackage{tikz}
\usetikzlibrary{svg.path}
\input{ORCID_logo}

\usepackage{hyperref}
\hypersetup{
    colorlinks=true,
    linkcolor={red!50!black},
    citecolor={blue!50!black},
    urlcolor={blue!20!black}
}
\usepackage{thmtools}
\usepackage[noabbrev,capitalize]{cleveref}

\input{CustomDefsPreamble}

\begin{document}

\input{CustomDefs}

\title[Blended backward error analysis]{Backward error analysis for conjugate symplectic methods}

\author[R~McLachlan]{Robert McLachlan}
\address[R~McLachlan]{Massey University, New Zealand}
\author[C~Offen]{Christian Offen$^{\ast}$ {\protect \orcidicon{0000-0002-5940-8057}}}
\address[C~Offen]{University of Paderborn, Germany}
\address{$^{\ast}$Corresponding author}
\email[corresponding author]{christian.offen@uni-paderborn.de}
\date{\today}

\begin{abstract}
The numerical solution of an ordinary differential equation can be interpreted as the exact solution of a nearby modified equation. Investigating the behaviour of numerical solutions by analysing the modified equation is known as backward error analysis.
If the original and modified equation share structural properties, then the exact and approximate solution share geometric features such as the existence of conserved quantities.
Conjugate symplectic methods preserve a modified symplectic form and a modified Hamiltonian when applied to a Hamiltonian system. 
We show how a blended version of variational and symplectic techniques can be used to compute modified symplectic and Hamiltonian structures. In contrast to other approaches, our backward error analysis method does not rely on an ansatz but computes the structures systematically, provided that a variational formulation of the method is known. 
The technique is illustrated on the example of symmetric linear multistep methods with matrix coefficients.
\end{abstract}

\maketitle

\section{Introduction}
\label{sec:introduction}

While the forward error of a numerical method compares the exact solution of an ODE with the numerical solution after one time-step $h$, to obtain qualitative statements about the long-term behaviour of numerical solutions to ODEs, it is helpful to consider a {\em modified ODE} whose exact solution closely approximates the numerical flow map at grid points. 
A modified equation can be obtained as an expansion of the numerical solution as a power series in the step-size $h$. Though the series does not converge in general, optimal truncation techniques have been established such that the flow of the modified system and the numerical method coincide at grid points up to exponentially small error terms. Computing and analysing the structural properties of modified equations is known as {\em backward error analysis (BEA)} (see, for instance, \cite[\S IX]{GeomIntegration} or \cite[\S 5]{Leimkuhler2005}). Next to the analysis of long-term behaviour of numerical schemes, backward error analysis has been used to improve the initialisation of multi-step methods \cite{ellison2014initializing} as well as to improve physics informed machine learning techniques \cite{symplecticShadowIntegrators,SymplecticShadowIntPoster,LagrangianShadowIntegrators}.

If a Hamiltonian ODE is discretised by a symplectic integrator, then any truncation of the modified equation is itself a Hamiltonian system with respect to the original symplectic structure and a modified Hamiltonian. These are also called {\em Shadow Hamiltonians}.
The existence of a modified Hamiltonian or a modified Lagrangian is a key ingredient to obtain statements about long-term behaviour of symplectic method, such as oscillatory energy errors over exponentially long time intervals. Moreover, just as in the exact system, symplectic symmetries of the modified system yield conserved quantities for the modified dynamics by Noether's theorem. This explains why symplectic integrators behave well on completely integrable systems. A detailed discussion can be found in \cite{GeomIntegration}. Vermeeren observed that backward error analysis for variational integrators can be done entirely on the Lagrangian side \cite{Vermeeren2017}.

In contrast to symplectic methods, 
conjugate symplectic methods preserve a modified symplectic structure rather than the original symplectic structure. Conjugate symplectic methods share the excellent long-term behaviour of symplectic methods. Moreover, Noether's theorem applies such that symmetries of the modified system yield modified conserved quantities of the modified dynamics. When modified structures are explicitly known, explicit expressions of modified conserved quantities can be derived. This motivates the development of techniques to compute modified symplectic structures and Hamiltonians.

While traditional methods for the computation of modified Hamiltonians use an Ansatz (i.e.\ an educated guess) of the Hamiltonian as a power series and match terms, working with an Ansatz is challenging when a modified Hamiltonian {\em and} a modified symplectic structure need to be computed simultaneously: the components of matrices representing symplectic structures are in this context not constant but depend on the state space variables, fulfil a symmetry condition, and satisfy the Jacobi identity, which is a partial differential equation \cite[\S VII.2]{GeomIntegration}. This makes finding a suitable Ansatz difficult.

A typical strategy \cite{GeomIntegration,MarsdenWestVariationalIntegrators} to obtain a structure preserving numerical method is to approximate the variational principle
\begin{equation}\label{eq:VarPrinc}
\delta S =0, \quad 
S(y) = \int_{t_0}^{t_N} L(y(t),\dot{y}(t)) \d t, \quad y(t_0) = y_0, y(t_N)=y_N
\end{equation}
which governs the exact Euler--Lagrange equations
\[
\frac{\d }{\d t} \frac{\p L}{\p \dot{y}} -  \frac{\p L}{\p y}  =0
\]
by a discrete variational principle
\[
\nabla S_\Delta (\{y_i\}_i) =0, \quad
S_\Delta(\{y_i\}_i) = \sum_{i=0}^{N-1} h L_\Delta(y_i,y_{i+1}).
\]
Since $y_0$ and $y_N$ are fixed in the variations considered in \eqref{eq:VarPrinc}, the gradient above is taken with respect to all inner grid points $y_1,\ldots,y_{N-1}$. We obtain
the discrete Euler--Lagrange equations
\[
\D_1 L_\Delta(y_i,y_{i+1}) + \D_2 L_\Delta(y_{i-1},y_{i}) =0, \; i=0,\ldots,N-1
\]
which yield approximations $y_i \approx y(t_0+ih)$ to an exact solution $y$. The term recursion is called a {\em variational method}.
Indeed, the class of variational methods is equivalent to the class of symplectic integrators \cite{GeomIntegration,MarsdenWestVariationalIntegrators}.

While backward error analysis for discrete Lagrangians $L_\Delta(y_i,y_{i+1})$ are established, discrete Lagrangians depending on several grid-points $L_\Delta(y_{i},y_{i+1},\ldots,y_{i+s})$ corresponding to multistep methods require different approaches because they are not symplectic but only preserve a modified symplectic structure. In other words, these method are conjugate to symplectic methods. However, the modified symplectic structures or conjugacies, respectively, are given by a formal power series that might not be convergent. Although rigorous optimal truncation results are not available, we will demonstrate in numerical examples that truncations can be useful objects in the analysis of the numerical methods.

In the following, we will introduce {\em blended backward error analysis} to systematically compute modified Hamiltonian and modified symplectic structures. We will prove the following theorem which applies, for instance, to series expansions $\mathcal{L}_\Delta$ of consistent discrete Lagrangians $L_\Delta(y(t),y(t+h),\ldots,y(t+sh))$.

\begin{theorem}\label{thm:Main}
Consider a power series $\mathcal{L}_\Delta(y^{[\infty]})$ in a formal variable $h$. The series depends on the jet $y^{[\infty]}=(y,\dot{y},\ddot{y},\ldots)$ of a variable $y$ such that any truncation only depends on a finite jet of $y$. Assume further that the truncation to zeroth order constitutes a regular Lagrangian $L(y,\dot{y})$, i.e.\ $\left(\frac{\p^2 L}{\p \dot{y}^i \p \dot{y}^j}\right)_{i,j}$ is invertible.
\begin{itemize}
	
	\item
	There exists a 2nd order modified equation given as a formal power series in $h$ such that for any $N \in \N$ a solution of the modified equation truncated to order $\mathcal{O}(h^N)$ solves the Euler--Lagrange equations to $\mathcal L^{[N]}_\Delta$ up to an error of order $\mathcal{O}(h^{N+1})$, where $\mathcal L^{[N]}_\Delta$ is the truncation of $\mathcal{L}_\Delta$ to order $\mathcal{O}(h^N)$.
	
	\item
	If we denote $z=(y,\dot{y})$, there exists a symplectic structure matrix $J_\mod$ and a Hamiltonian $H_\mod$ given as formal power series in $h$ such that solutions to Hamilton's equations
	\[
	\dot{z} = J^{[N]}_\mod(z)^{-1} \nabla H^{[N]}_\mod(z)
	\]
	fulfil the modified equation. Here, $J^{[N]}_\mod$ and $H^{[N]}_\mod$ are truncations to order $\mathcal{O}(h^N)$ of $J_\mod$ and $H_\mod$, respectively such that $\mathcal{L}^{[N]}_\Delta$ is regular, i.e.\ $\left(\frac{\p^2 \mathcal{L}^{[N]}_\Delta}{\p {y^{(M)}}^i \p {y^{(M)}}^j}\right)_{i,j}$ is invertible, where $y^{(M)}$ is the highest derivative of $\mathcal{L}_\Delta^{[N]}$.
\end{itemize}
\end{theorem}

The technique will be illustrated for linear multistep methods with matrix valued coefficients (\cref{sec:MSApplication}). These occur, for instance, when in a system of coupled ODEs components of the differential equation are discretised separately with traditional linear multistep method, when multistep methods are stabilised \cite{Jodar1991,Lambert1972,Sigurdsson1971}, or when discretisation schemes of PDEs are analysed \cite{BEASymPDE}. 

Moreover, we will analyse under which conditions modified Lagrangians exist: if the original equation is the Euler--Lagrange equation to a variational principle of the form \eqref{eq:VarPrinc} for a Lagrangian $L(y,\dot{y})$, it is natural to ask, whether there exists a modified Lagrangian $L_\mod(y,\dot{y})$ such that the modified equation is governed by the Euler--Lagrange equations to $L_\mod$. In contrast to classical variational integrators, for which $L_\mod(y,\dot{y})$ is known to exist and can be computed \cite{Vermeeren2017}, for conjugate symplectic schemes $L_\mod$ may only exists in modified variables $L_\mod(\tilde{y},\dot{\tilde{y}})$. We will prove that $L_\mod$ exists in the original variables $(y,\dot{y})$ if $J_\mod$ has the form
\[
J_\mod = \begin{pmatrix}
	\ast & \ast\\
	\ast & 0
\end{pmatrix}.
\]
In particular, this shows that for classical consistent, symmetric, linear multistep methods with matrix coefficients with a central force evaluation applied to second order equations $L_\mod$ exists in the original variables $(y,\dot{y})$ if all coefficients are multiples of the identity matrix, i.e.\ we have a multistep method with scalar coefficients. However, in the general case of matrix coefficients $L_\mod$ only exists in modified variables $(\tilde{y},\dot{\tilde{y}})$.

The article is structured as follows: \Cref{sec:MSApplication} illustrates the ideas of blended backward error analysis on linear multi-step methods. \Cref{sec:ModComp} shows how to compute the modified data $J_\mod$ and $H_\mod$ introduced in \cref{thm:Main}. The technique is then applied to linear multi-step methods with matrix valued coefficients in \cref{sec:CompEx} and results are illustrated by numerical experiments.
Additionally, for comparison of blended backward error analysis with classical backward error analysis, \cref{sec:IlluBEA} contains an application of blended backward error analysis to a mechanical ordinary differential equation discretised with the Störmer-Verlet scheme.
A formal proof of \cref{thm:Main} is provided in \cref{sec:proofs}. Finally, \cref{sec:ExistenceLmod} discusses the existence of modified Lagrangians as formal power series and future research directions are suggested in \cref{sec:future}.

\section{Application of blended backward error analysis to linear multistep methods with matrix coefficients}\label{sec:MSApplication}

To illustrate the idea of blended backward error analysis, we compute modified symplectic structures and Hamiltonians of linear multistep methods. 

Consistent, symmetric linear multistep methods with a single force evaluation applied to the second order ODE
\begin{equation}
\label{eq:ODE}
\ddot{y} = \nabla U(y(t))
\end{equation} take the form
\begin{equation}
\label{eq:MS}
\sum_{j=1}^{\frac{s}{2}} A_j(y(t-jh)-2y(t)+y(t+jh)) = h^2 \nabla U(y(t))
\end{equation}
with \begin{equation}
\label{eq:consistency}
\sum_{j=1}^{\frac{s}{2}} j^2 A_j = I.
\end{equation}
Relation \eqref{eq:consistency} is coming from the consistency requirement \cite{Jodar1991}. These are $s$-step methods, where $s$ is even.
Here we allow matrix valued coefficients $A_j$ \cite{Jodar1991,Lambert1972,BEASymPDE,Sigurdsson1971}, $h$ is the step-size and $I$ denotes the identity matrix.
If the coefficients $A_j$ are scalars, then the schemes constitute classical consistent, symmetric linear multistep methods.
A series expansion of \eqref{eq:MS} in $h$ is equivalent to a power series expression of the form
\begin{equation}
\label{eq:SeriesMS}
\ddot{y} = \nabla U(y(t)) + \sum_{i=1}^N h^i \tilde{g}_i(y(t),\ldots,y^{(a_i)}) + \mathcal{O}(h^{N+1})
\end{equation}
for some $h$-independent expressions $\tilde g_i$ in the $a_i$-jet of $y$ at $t$. Substituting 3rd and higher derivatives on the right hand side of \eqref{eq:SeriesMS} with derivatives of \eqref{eq:SeriesMS} itself iteratively yields an equation of the form
\begin{equation}
\label{eq:modODE}
\ddot{y} = \nabla U(y(t)) + \sum_{i=1}^N h^i g_i(y(t),\dot{y}(t)) + \mathcal{O}(h^{N+1})
\end{equation}
which is called the {\em modified equation} of method \eqref{eq:MS} applied to \eqref{eq:ODE}.
We refer to \cite{GeomIntegration} for optimal truncation techniques and a discussion of spurious solutions not covered by the considered modified system for the case of linear multistep methods with scalar coefficients.
In the following, we will focus on the question which structural properties the modified equation \eqref{eq:modODE} shares with the original ODE \eqref{eq:ODE}.

\subsection*{Variational Structure}
The ODE \eqref{eq:ODE} has first order variational structure as it is the Euler--Lagrange equation
\[
\frac{\d}{\d t} \frac{\p L}{\p \dot{y}} -\frac{\p L}{\p y} =0
\]
to the variational principle
\[
\delta S =0
\quad
\text{for}
\quad
S(y) = \int L(y(t),\dot{y}(t))\d t
\]
with
\[
L(y,\dot{y}) = \frac{1}{2} \|\dot{y}\|^2 + U(y).
\]
Moreover, there exists a variational principle for \eqref{eq:MS}:

\begin{lemma}
Let the matrices $A_j \in \R^{n \times n}$ be symmetric. For $T>0$ let $\mathbb{T}$ either be the circle $\mathbb{T}=\R / T \Z$ or the real line $\mathbb{T}=\R$.
For $y$ defined on $\mathbb{T}$ the variational principle
\begin{equation}
\label{eq:SDelta}
\delta S_\Delta =0
\quad
\text{for}
\quad
S_\Delta(y) = \int_{\mathbb{T}} L_{\Delta}\left(y(t),y(t+h),\ldots,y\left(t+\frac{s}{2}h\right)\right)\d t
\end{equation}
with
\begin{align*}
L_{\Delta}\left(y(t),y(t+h),\ldots,y(t+\frac{s}{2}h)\right)\qquad \qquad \qquad \qquad \qquad \qquad \qquad \qquad\\ 
= \frac{1}{2h^2}\left( \sum_{j=1}^{\frac{s}{2}}  \langle A_j (y(t+jh)-y(t)),y(t+jh)-y(t)\rangle \right) + U(y(t)).
\end{align*}
implies the functional equation \eqref{eq:MS}.
Moreover, if $\mathbb{T}=[a,b]$ is an interval, then \eqref{eq:SDelta} implies \eqref{eq:MS} on the interval $\accentset{\circ}{\mathbb{T}}=  [a+\frac{s}{2}h,b-\frac{s}{2}h]$.
Here we assume that the function space for $y$ and the potential $U$ are such that $U \circ y$ and $\nabla U \circ y$ constitute square integrable functions.
\end{lemma}

\begin{proof}
Let $\Delta_{\tau} y$ denote the forward difference, i.e.\ $(\Delta_{\tau} y)(t) = y(t+\tau)-y(t)$. Then $\langle y,\Delta_{\tau} z \rangle_{L^2(\mathbb{T},\R^n)} = \langle \Delta_{\tau}^\ast y, z \rangle_{L^2(\mathbb{T},\R^n)}$ holds with $\Delta_{\tau}^\ast = \Delta_{-\tau}$ on $\mathbb{T}$ if $\mathbb{T} \in \{\R, \R / T\Z\}$ or on the interval $[a+\tau,b-\tau]$ if $\mathbb{T}=[a,b]$.
The expression $\Delta_{\tau}^\ast \Delta_{\tau} y$ corresponds to the central difference
\[ (\Delta_{\tau}^\ast \Delta_{\tau} y) (t) = -y(t+\tau) + 2 y(t) - y(t-\tau).\]

Let $\delta$ denote the variational derivative in the direction of a variation $\delta y$, i.e.\
$\delta S_\Delta (y) =  \lim_{\e \to 0} \frac{1}{\e} (S_\Delta(y+\e \delta y)-S_\Delta(y))$.
Let $\delta y \in \mathcal{C}_c^\infty(\mathbb{T},\R^n)$, if $\mathbb{T} \in \{\R, \R / T\Z\}$ and let $\delta y \in \mathcal{C}_c^\infty(\accentset{\circ}{\mathbb{T}},\R^n)$, if $\mathbb{T} =[a,b]$.
We have
\begin{align*}
\delta S_\Delta (y)
&=  \frac 1 {2h^2} \sum_{j=1}^{\frac s2} \delta  \langle A_j \Delta_{jh} y, \Delta_{jh} y  \rangle_{L^2(\mathbb{T},\R^n)}
+ \delta \int_{\mathbb{T}} U(y(t)) \, \d t\\
&=  \frac 1 {h^2} \sum_{j=1}^{\frac s2}   \langle A_j \Delta_{jh} y, \Delta_{jh} \delta y  \rangle_{L^2(\mathbb{T},\R^n)} + \langle \nabla U(y),  \delta y  \rangle_{L^2(\mathbb{T},\R^n)}\\
&=  \frac 1 {h^2} \sum_{j=1}^{\frac s2}   \langle A_j \Delta_{jh}^\ast\Delta_{jh} y + \nabla U(y),  \delta y  \rangle_{L^2(\mathbb{T},\R^n)}.
\end{align*}
	
Now \eqref{eq:MS} follows from the fundamental lemma of the calculus of variations on $\mathbb{T}$ or $\accentset{\circ}{\mathbb{T}}$, respectively.
\end{proof}

To analyse structure preserving properties of the method \eqref{eq:MS}, it might seem natural to seek a modified Lagrangian $L_{\mod}(y,\dot{y})$ given as a formal power series in the step-size $h$ such that the modified variational principle
\[
\delta S_\mod =0
\quad
\text{for}
\quad
S_\mod(y) = \int L_\mod(y(t),\dot{y}(t))\d t
\]
covers smooth solutions of \eqref{eq:SDelta} up to any order in the step-size $h$. However, we show that although a 1st order Lagrangian $L_{\mod}$ covering the modified equations always exists as a power series, it only exist in modified variables $(\tilde y,\dot{\tilde y})$ in the most general case. Even for simple methods, the existence of an expression in closed form for the change of coordinates from $(y,\dot y)$ to $(\tilde y,\dot{\tilde y})$ can not be expected. This makes it difficult to compute $L_{\mod}$ using an ansatz.

\subsection*{Hamiltonian structure} 
Another approach is to work on the Hamiltonian side. The ODE \eqref{eq:ODE} has the form of a Hamiltonian system
\begin{equation}
\label{eq:HamiltonianODE}
\dot z(t) = J^{-1} \nabla H(z(t)), \quad H(z) = \frac{1}{2}\|\dot{y}\|^2 - U(y), \quad z=\begin{pmatrix}y\\ \dot y\end{pmatrix}.
\end{equation}
Here
\begin{equation}
\label{eq:J}
J = \begin{pmatrix}
0&-I\\
I & 0
\end{pmatrix}
\end{equation}
is the standard symplectic structure. 

In this paper we use a blended approach of the variational and Hamiltonian viewpoint to systematically compute a modified Hamiltonian system
\begin{equation}
\label{eq:HamModODE}
\dot z(t) = J^{-1}_\mod(z(t)) \nabla H_\mod(z(t))
\end{equation}
consisting of a modified Hamiltonian $H_\mod$ and a modified symplectic structure $J_\mod$ given as formal power series in $h$ such that for a truncation to arbitrary order $N$ the ODE \eqref{eq:HamModODE} covers the modified equation \eqref{eq:modODE} up to higher order terms.
Here $J_\mod$ is a skew symmetric matrix which satisfies a Jacobi identity. 
The following theorem can be considered as an instance of \cref{thm:Main} and will be proved in \cref{sec:proofs}.
\begin{theorem}\label{thm:HJmod}
Let $N \in \N$ denote the considered order of the series expansion of the matrix multistep method \eqref{eq:MS}, where the coefficients are symmetric matrices.
There exists a modified symplectic structure $J^{[N]}_\mod$, which is $\mathcal{O}(h)$ close to $J$ and a modified Hamiltonian $H^{[N]}_\mod$, which is $\mathcal{O}(h)$ close to $H$, such that 
\begin{equation}
\label{eq:HamModODE2}
\dot z(t) = (J^{[N]}_\mod(z))^{-1} \nabla H^{[N]}_\mod(z(t))
\end{equation}
with coordinate $z = (y,\dot y)$ is equivalent to the modified equation \eqref{eq:modODE} up to terms of order $\mathcal{O}(h^{N+1})$.
\end{theorem}
Here $O(h)$-closeness of $J^{[N]}_\mod$ and $J$ means that the zeroth coefficient of the polynomial $J^{[N]}_\mod$ in the formal variable $h$ is given by $J$. $O(h)$-closeness of $H^{[N]}_\mod$ and $H$ has an analogous meaning.

We observe conditions under which a modified first order Lagrangian $L_{\mod}(y,\dot{y})$ exists in the original variable $y$ by analysing the modified symplectic structure $J_\mod$.

\begin{theorem}\label{thm:Smod}
If the matrices $A_j$ in \eqref{eq:MS} are scalar multiples of the identity matrix, then there exists a modified Lagrangian $L^{[N]}_\mod$ depending on $(y,\dot{y})$ that is $\mathcal{O}(h)$-close to $L(y,\dot{y})$ such that the modified variational principle 
\[
\delta S^{[N]}_\mod =0
\quad
\text{for}
\quad
S^{[N]}_\mod(y) = \int L^{[N]}_\mod(y(t),\dot{y}(t))\d t
\]
yields the modified equation \eqref{eq:modODE} up to terms of order $\mathcal{O}(h^{N+1})$.
\end{theorem}

Again, $O(h)$-closeness is to be interpreted in a formal sense, analogously to its meaning in \cref{thm:HJmod}.

\Cref{thm:Smod} applies to traditional multistep methods with scalar coefficients. However, we will see that for general linear multistep methods with matrix-valued coefficients the existence of a first order modified Lagrangian in the original variable $y$ cannot be expected. A proof of \cref{thm:Smod} is postponed to \cref{sec:ExistenceLmod}.

\section{Computation of modified Hamiltonian structure}\label{sec:ModComp}

In this section we introduce a method to compute the modified data $J^{[N]}_\mod$ and $H^{[N]}_\mod$ of \cref{thm:HJmod} such that \eqref{eq:HamModODE2} governs \eqref{eq:modODE}. We will then verify the validity of the construction method and prove \cref{thm:Main} and \ref{thm:HJmod} in \cref{sec:proofs}.

Let $\mathcal{L}^{[N]}_\Delta$ denote the series expansion of
\[L_{\Delta}\left(y(t),y(t+h),\ldots,y\left(t+\frac{s}{2}h\right)\right)\]
to order $N$ in the step-size $h$. The expression $\mathcal{L}^{[N]}_\Delta$ depends on the $M$-jet of $y$ at $t$ for some $M \in \N$. Notice that the order $M$ variational principle
\begin{equation}
\label{eq:SDeltaM}
\delta \mathcal{S}^{[N]}_\Delta(y) =0 \quad \text{with} \quad \mathcal{S}^{[N]}_\Delta(y) = \int \mathcal{L}^{[N]}_\Delta(y(t),\dot{y}(t),\ldots,y^{(M)}(t))\,\d t
\end{equation}
recovers \eqref{eq:SeriesMS} up to higher order terms. We first compute a high-dimensional Hamiltonian system defined on the $2M-1$-jet space of $y$ corresponding to the order $M$ variational principle \eqref{eq:SDeltaM}. The Hamiltonian principle is then reduced to a Hamiltonian system defined on the 1-jet space of $y$. It has the form \eqref{eq:HamModODE} and covers \eqref{eq:modODE} up to higher order terms.


To construct the high-dimensional Hamiltonian system, we use Ostrogradsky's Hamiltonian description of high-order Lagrangian systems \cite{whittaker1988}. For this we define variables
\[q = (y,\dot y,\ldots, y^{(M-1)})\] and for $i=1,\ldots,M$
\[
p_i^j 
= \sum_{k=0}^{M-i} (-1)^k \frac{\d^{k}}{\d t^{k}}
\frac{\p \mathcal{L}^{[N]}_\Delta}{\p (y_j)^{(k+i)}}.
\]
Here the index $j$ enumerates the components of $y$ and $\frac{\d}{\d t}$ denotes the total derivative operator on the jet-space of $y$, which acts like
\[
\frac{\d}{\d t} \rho\left(y,\dot y,\ldots,y^{(M)}\right)
= \sum_{i=0}^{M} \left\langle \nabla_{y^{(i)}}\rho, y^{(i+1)}\right\rangle.
\]
on a scalar valued function $\rho$ defined on the $M$-jet space of $y$.
The high dimensional Hamiltonian system consists of the Hamiltonian
\begin{equation}
\label{eq:Htilde}
\mathcal H^{[N]} = \sum_{i=1}^M \langle p_i , \dot{q}^i \rangle - \mathcal{L}^{[N]}_\Delta,
\end{equation}
where all expressions are expressed in $y^{[2M-1]}=(y,\dot y,\ldots,y^{(2M-1)})$, and the symplectic structure matrix $\mathcal{J}^{[N]}_\mod(y^{[2M-1]})$. The skew-symmetric matrix $\mathcal{J}^{[N]}_\mod(y^{[2M-1]})$ is the representing matrix of the differential 2-form
\begin{equation}
\label{eq:Omega}
\Omega^{[N]}=\sum_{i=1}^M \sum_{j=1}^n \d p_i^j \wedge \d q^i_j,
\end{equation}
where $p_i^j$ and $q^i_j$ are interpreted as functions in the variable $y^{[2M-1]}$ of the $2M-1$-jet space, i.e.\ $\mathcal{J}^{[N]}_\mod$ is the anti-symmetrised tensor product\footnote{This corresponds to the command {\tt TensorWedge} in Wolfram Mathematica.} $\wedge$ of the gradients $\nabla_{y^{[2M-1]}} p_j^i$ and $\nabla_{y^{[2M-1]}} q^i_j$ summed over all indices.

To compute the modified Hamiltonian system on the 1-jet space with variable $y^{[1]}=(y,\dot{y})$, the variables $y^{(2)},\ldots,y^{(2M-1)}$ in the expression \eqref{eq:Htilde} for the Hamiltonian $\mathcal H^{[N]}(y^{[2M-1]})$ are repeatedly replaced by \eqref{eq:modODE} until higher derivatives only occur in $\mathcal{O}(h^{N+1})$ terms. This yields $H^{[N]}_\mod(y,\dot{y})$. Similarly, we can consider $p_i^j$ and $q^i_j$ as functions of $(y,\dot{y})$ truncating $\mathcal{O}(h^{N+1})$ terms. The matrix $J^{[N]}_\mod$ is then given as the representing matrix of the 2-form $\Omega^{[N]}$ pulled to the 1-jet space of the variable $y$, i.e.\ interpreting $y^{[1]}=(y,\dot{y})$ as the only independent variables.
Equivalently, $J^{[N]}_\mod$ is the anti-symmetrised tensor product $\wedge$ of the gradients $\nabla_{y^{[1]}} p_j^i$ and $\nabla_{y^{[1]}} q^i_j$ summed over all indices. (As $J^{[N]}_{\mod}$ is constructed from a closed differential 2-form, it is automatically skew-symmetric and satisfies the Jacobi identity.)
This completes the construction of the modified data.

The system \eqref{eq:HamModODE} recovers \eqref{eq:modODE} up to higher order terms as we will prove after a computational example.

\section{Computational example}\label{sec:CompEx}

As introduced in \cref{sec:MSApplication}, consider the multistep method
\begin{equation}\label{eq:MSEx}
A_2 y(t-2h) + A_1 y(t-h) - 4y(t) + A_1 y(t+h) + A_2 y(t+2h) = h^2\nabla U(y(t))
\end{equation}
with matrix coefficients in dimension $n=2$. By the consistency requirement, $A_2= 1/4(I-A_1)$. We obtain
\[
J_\mod^{[4]} = J 
+ h^2 \begin{pmatrix}
J^2_{11} & -J^2_{21}\\
J^2_{21} & 0
\end{pmatrix}
+ h^4 \begin{pmatrix}
J^4_{11} & -J^4_{21}\\
J^4_{21} & J^4_{22}
\end{pmatrix}
\]
with
\[
J^2_{11} = \begin{pmatrix}
0&-b_1\\
b_1&0\\
\end{pmatrix}
\]
where
\begin{align*}
b_1&=\frac{1}{4} \Big(\dot{y}_2 \left(a_{12}(U^{(2,1)}-U^{(0,3)}) +\left(a_{22}-a_{11}\right) U^{(1,2)}\right)\\
&+\phantom{\frac{1}{4} \Big(}\dot{y}_1 \left(-a_{12}(U^{(2,1)}-U^{(0,3)})
+\left(a_{22}-a_{11}\right)U^{(2,1)} \right)\Big).
\end{align*}
Here and in the following $U^{(k,l)} = \frac{\p^{k+l}U(y)}{\p^k y_1 \p^l y_2}$.
As the expressions of higher order terms become quite complicated, we refer the reader to the Mathematica Notebooks of our accompanying source code \cite{ConjugateSymplecticSoftware}. However, as this will be relevant in the discussion later, we are reporting $J^4_{22}$ for the special case that $A_1$ is a diagonal matrix: we have
\[
J^4_{22} = \begin{pmatrix}
0&-b_2\\
b_2&0\\
\end{pmatrix}
\]
with
\[
b_2= \frac{1}{8}(a_{11}^2-a_{22}^2 + 2(a_{22}-a_{11}))(U^{(2,1)} \dot{y}_1 + U^{(1,2)} \dot{y}_2) \quad \text{if} \quad A_1 = \begin{pmatrix}
a_{11}&0\\ 0& a_{22}
\end{pmatrix}.
\]

The modified Hamiltonian $H^{[4]}_\mod$ for the general case is given as
\[
H^{[4]}_\mod(y,\dot y)
= H(y,\dot y) + h^2 H_2(y,\dot y) + h^4 H_4(y,\dot y)
\]
with
\begin{align*}
H_2(y,\dot y) &=
\frac{1}{24} \Big((\dot{y}_1){}^2 (2 (4-3 a_{11}) U^{(2,0)}-6 a_{12} U^{(1,1)})\\
&-2 \dot{y}_2 \dot{y}_1 (3 a_{12} U^{(0,2)}+(3 a_{11}+3 a_{22}-8) U^{(1,1)}+3 a_{12} U^{(2,0)})\\
&-6 a_{22} U^{(0,2)} (\dot{y}_2){}^2-6 a_{12} U^{(1,1)} (\dot{y}_2){}^2\\
&+(3 a_{22}-4) (U^{(0,1)})^2+3 a_{11} (U^{(1,0)})^2\\
&+6 a_{12} U^{(0,1)} U^{(1,0)}+8 U^{(0,2)} (\dot{y}_2){}^2-4 (U^{(1,0)})^2\Big).
\end{align*}

For further terms, we refer the reader to the Mathematica Notebooks of our accompanying source code.
Hamilton's equations
\[
\dot{z} = (J^{[4]}_\mod(z))^{-1} \nabla H^{[4]}_\mod(z)
\]
for the modified Hamiltonian system are equivalent to the modified equation \eqref{eq:modODE} truncating terms of order $\mathcal{O}(h^6)$.

\Cref{fig:MSEx} shows a numerical experiment with a rotational invariant potential.
The start values for the multistep formula were obtained using the fourth order modified equation. 
Trajectories computed with the multistep scheme look very regular.
The quantities $H^{[0]}_\mod =H$, $H^{[2]}_\mod$, and $H^{[4]}_\mod$ evaluated along a trajectory show oscillatory energy error behaviour. Experiments with different values for the step-size $h$ confirm the preservation of $H^{[k]}_\mod$, $k=0,2,4$ up to truncation error.
Initialising the multistep scheme with the fourth order modified equation, the effects of spurious solutions was minimised. However, as $h$ is decreased, spurious solutions cause wriggles in the energy error $H^{[k]}_\mod-H(z_{\mathrm{init}})$, $k=0,2,4$ which eventually prevent further energy error decay. 

\begin{figure}
	\includegraphics[width=0.45\linewidth]{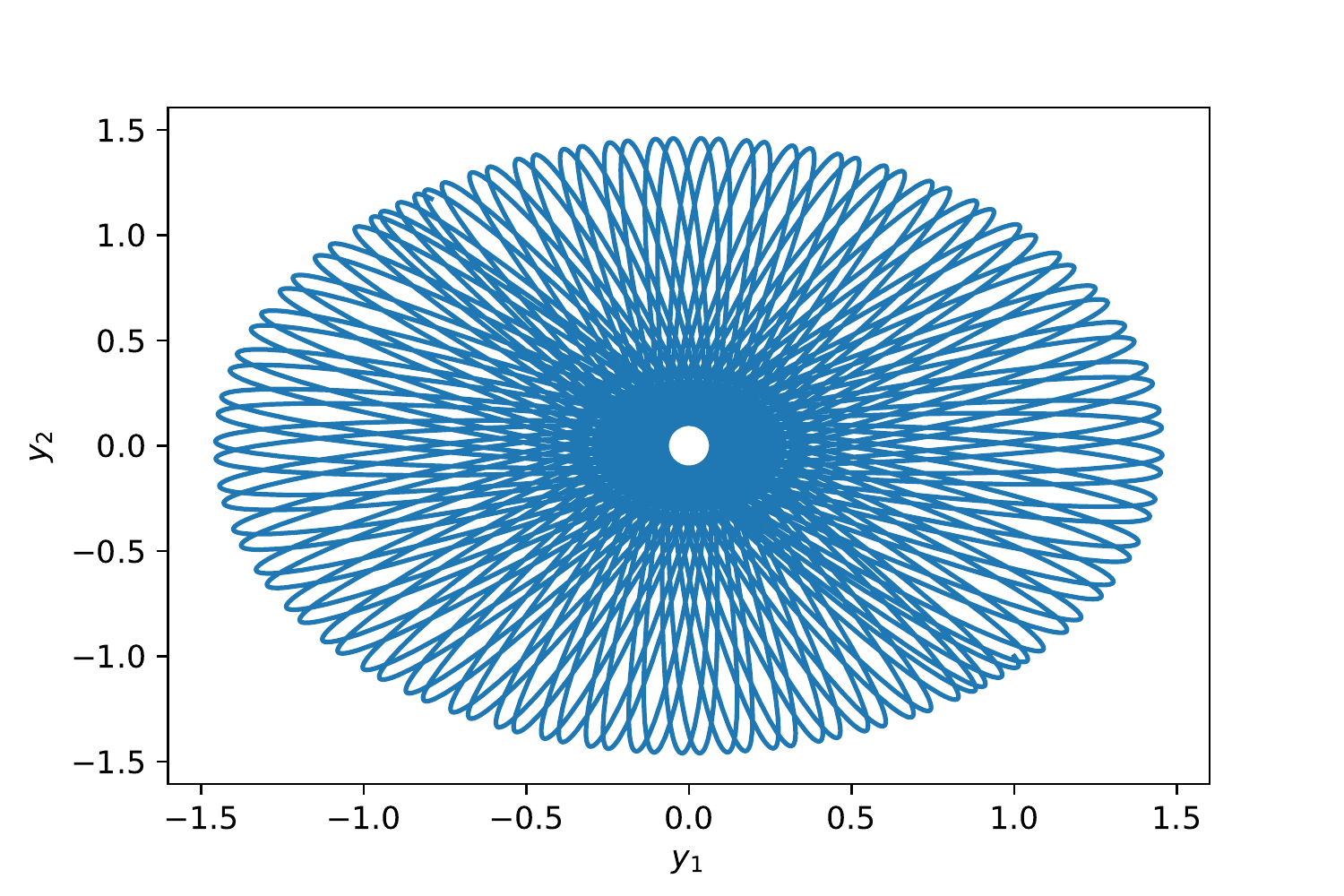}\\
	\includegraphics[width=0.45\linewidth]{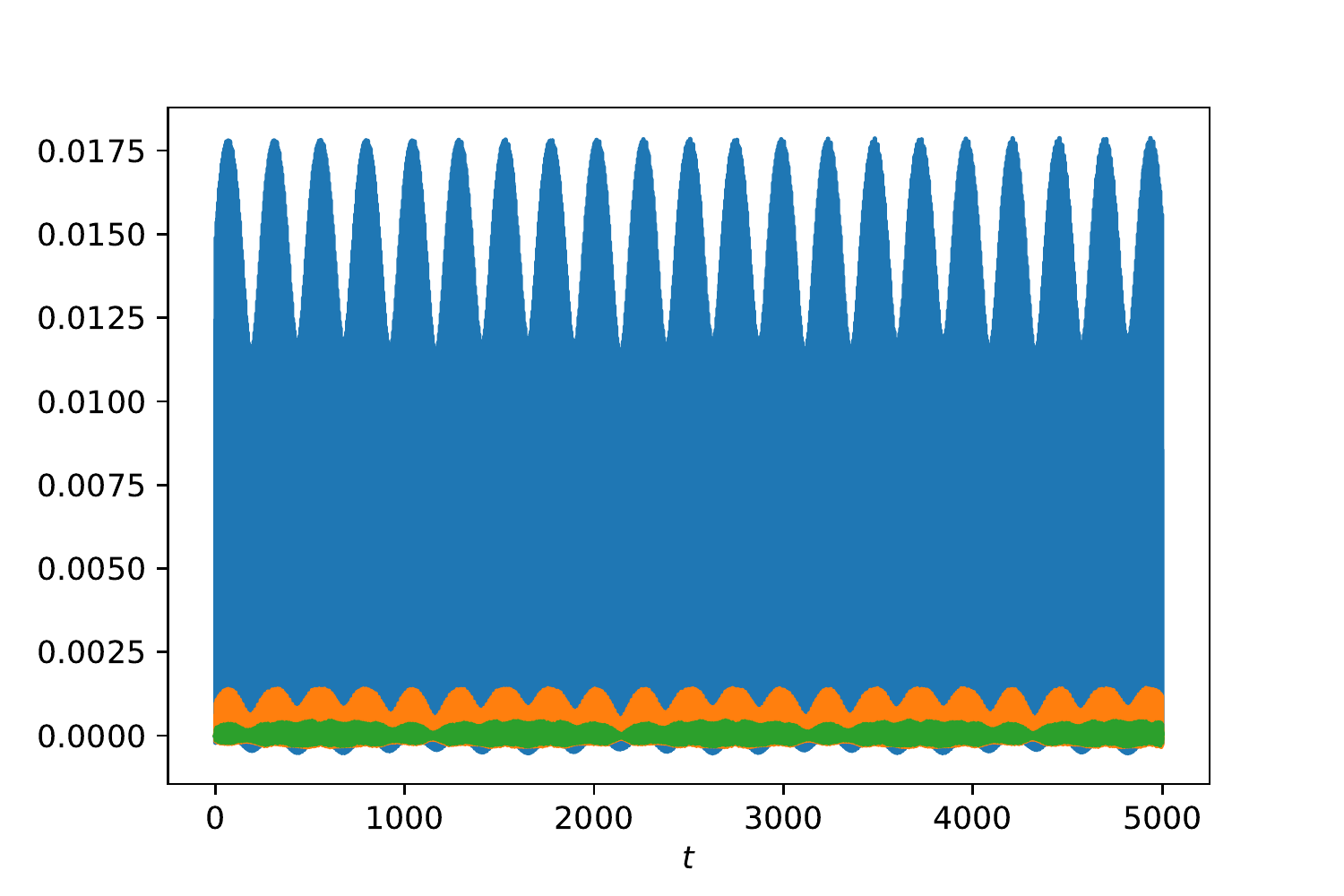}
	\includegraphics[width=0.45\linewidth]{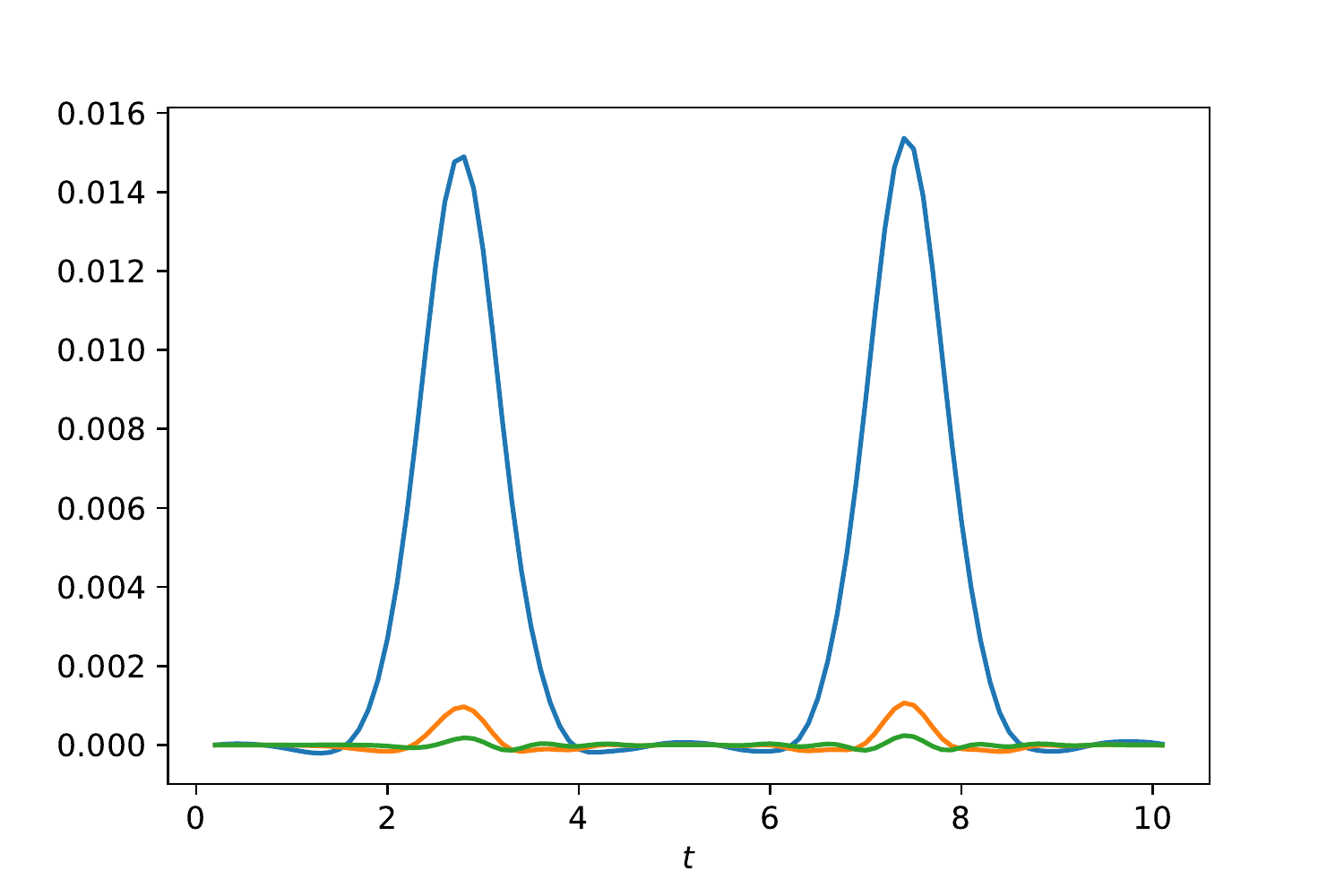}
	\caption{Numerical experiment with the multistep scheme \eqref{eq:MSEx} with $A_2 = \diag(0.85,1.25)$, $A_1 = I-4A_2$, $U(y) = \exp\left(-\frac{1}{2}(y_1^2 + y_2^2)\right)$ and time-step $h=0.1$.
		The multistep formula was initialised at times $0$, $h$, $2h$, $3h$ by integrating the order 4 modified equation using very fine Euler-steps starting from the initial value $(y_{\mathrm{init}},\dot{y}_{\mathrm{init}})=((1,-1),(0.1,-0.2))$.
		The figure at the top shows a trajectory up to time $t=500$, which looks like an orbit in a completely integrable system. 
		The figures below show an evaluation along the trajectory of $H-H((y_{\mathrm{init}},\dot{y}_{\mathrm{init}}))$ (blue) as well as $H^{[2]}_\mod-H^{[2]}_\mod((y_{\mathrm{init}},\dot{y}_{\mathrm{init}}))$ (orange) and $H^{[4]}_\mod-H^{[4]}_\mod((y_{\mathrm{init}},\dot{y}_{\mathrm{init}}))$ (green) up to time $t=5000$ and $t=10$, respectively. We see oscillatory energy error behaviour.
	}\label{fig:MSEx}
\end{figure}


If $A_2 = \diag(\alpha,\alpha)$ with $\alpha \ge \frac{1}{4}$ and $A_1 = I-4A_2$, then \eqref{eq:MSEx} corresponds to a classical stable\footnote{A multistep scheme for second order equations is stable if all roots of its generating polynomial lie in the closed unit disk and those on the circle are at most double zeros \cite[XV.1.2]{GeomIntegration}.} multistep scheme: the generating polynomial $\rho$ to \eqref{eq:MSEx} is given as
\[
\rho(\xi)=\alpha +(1-\alpha) \xi -2(1-3 \alpha) \xi^2+ (1-4 \alpha) \xi^3 + \alpha \xi^4.
\]
The polynomial $\rho$ has a double root at 1 as well as the roots
\[
\xi_{3} = 1-\frac{1}{2 \alpha} + \frac{\sqrt{1-4\alpha}}{2\alpha}, \quad \xi_{4} = 1-\frac{1}{2 \alpha} - \frac{\sqrt{1-4\alpha}}{2\alpha}.
\]
Since $\alpha \ge \frac{1}{4}$, the roots $\xi_3$ and $\xi_4$ are complex conjugate to each other and lie on the unit circle such that the scheme is stable.

Moreover, $L_\Delta$ and $\mathcal{L}_\Delta$ are rotationally invariant because $A_1$ and $A_2$ commute with rotation matrices. An application of Noether's theorem to $\mathcal{L}^{[N]}_\Delta$ yields the following modified angular momentum:
\[
\mathcal{I}^{[N]} = \sum_{m=1}^M \sum_{k=0}^{m-1} (-1)^k \left\langle \nabla_{y^{k}} \mathcal{L}^{[N]}_\Delta , \d R y^{(m-1-k)} \right\rangle, \quad \d R=\begin{pmatrix}0&-1\\1&0\end{pmatrix}
\]
The integer $M$ is the order of the highest derivative of $y$ in $\mathcal{L}^{[N]}_\Delta$. For the truncation order $N=4$ we have $M=5$. Using \eqref{eq:modODE} repeatedly, derivatives of $y$ of order greater than two are replaced by terms in $y$, $\dot{y}$ and $\mathcal{O}(h^6)$-terms. After truncation at order 4 in $h$ we obtain the modified angular momentum $I^{[4]}_\mod(y,\dot{y})$. \Cref{fig:MSExDiag} shows the conservation error of $I=I^{[0]}_\mod$, $I^{[2]}_\mod$, $I^{[4]}_\mod$ along a trajectory. We see oscillatory error behaviour. The oscillations shrink as the step-size $h$ decreases until spurious solutions prevent further decrease.

\begin{figure}
	\begin{center}
		\includegraphics[width=0.45\linewidth]{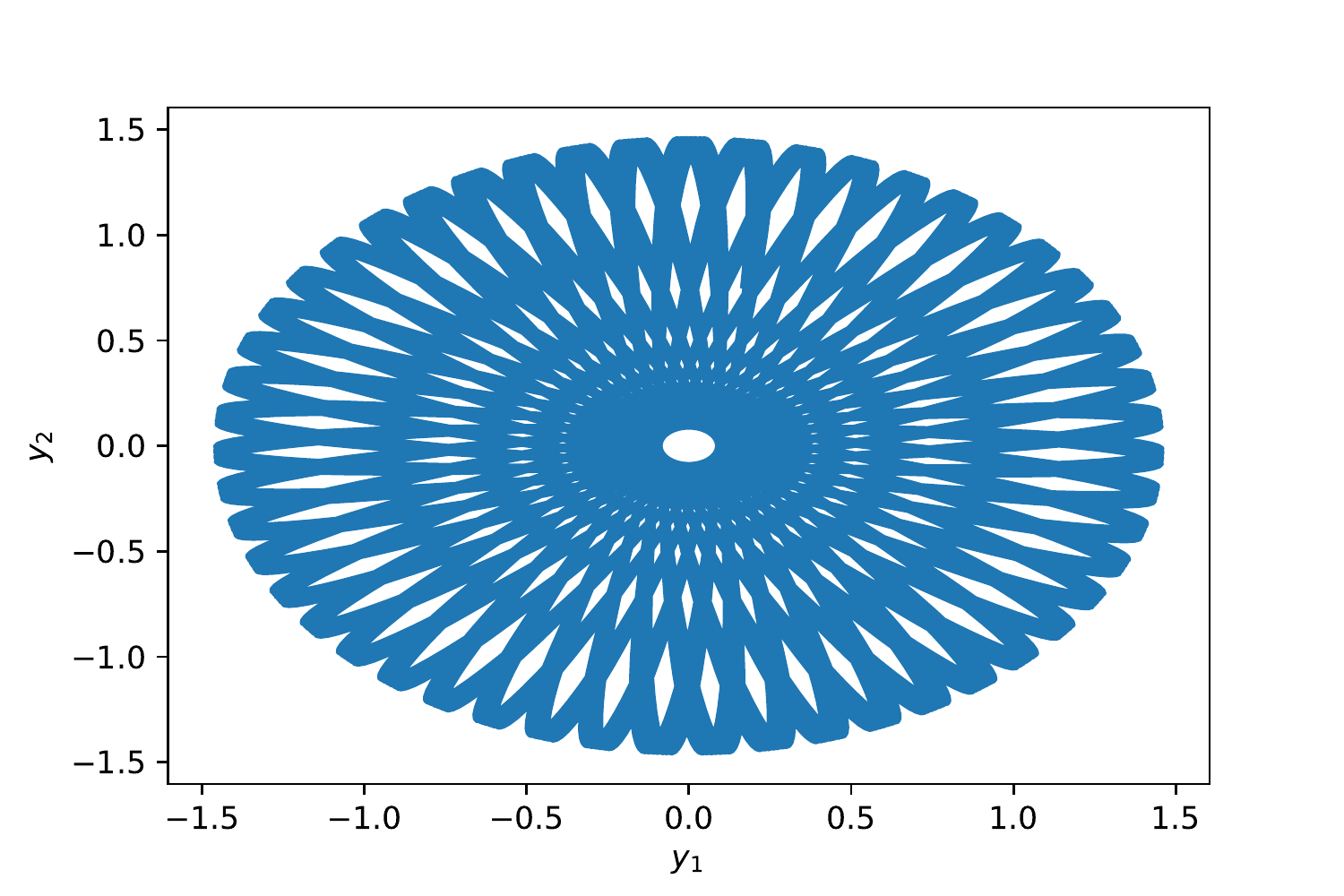}\\
		\includegraphics[width=0.45\linewidth]{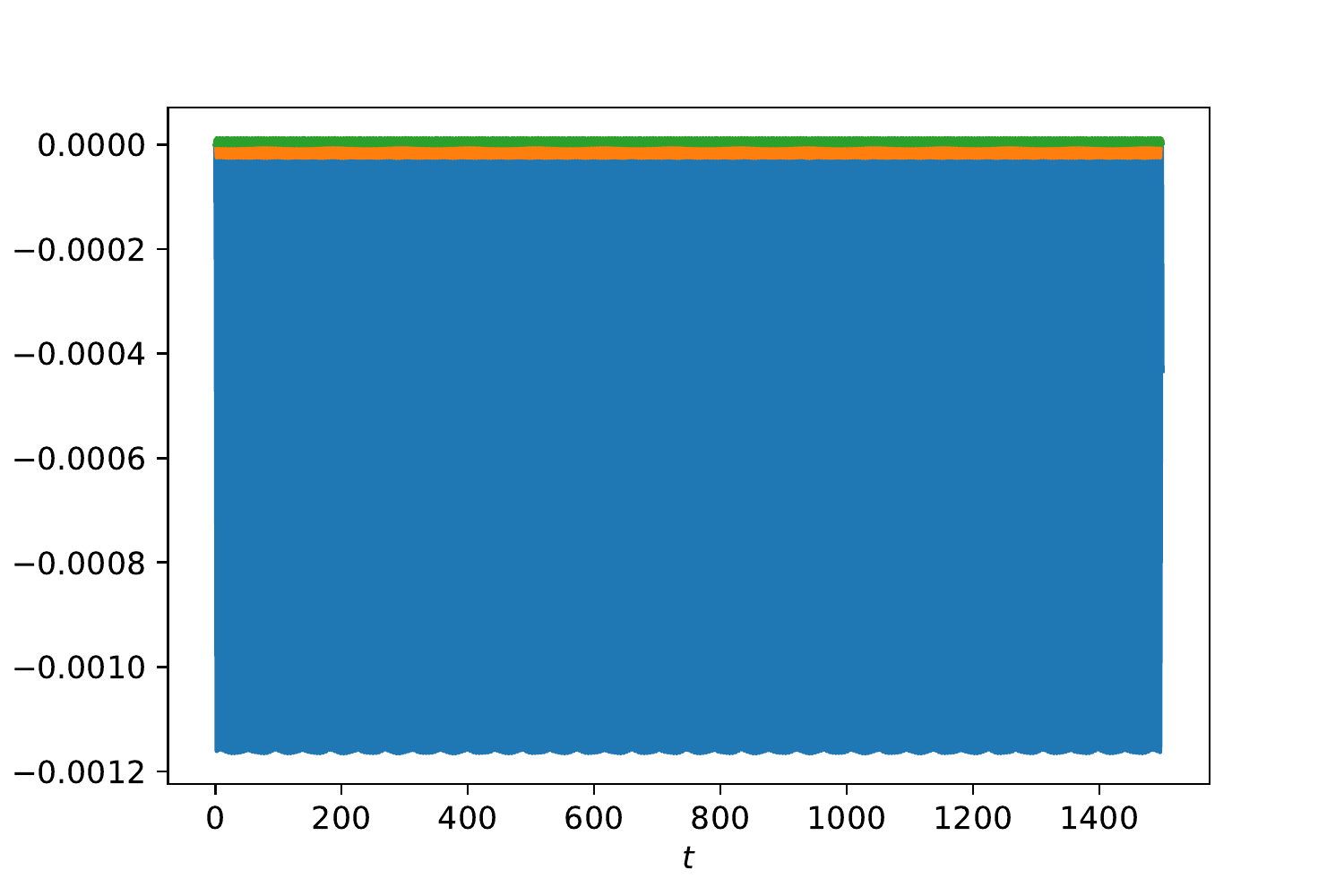}
		\includegraphics[width=0.45\linewidth]{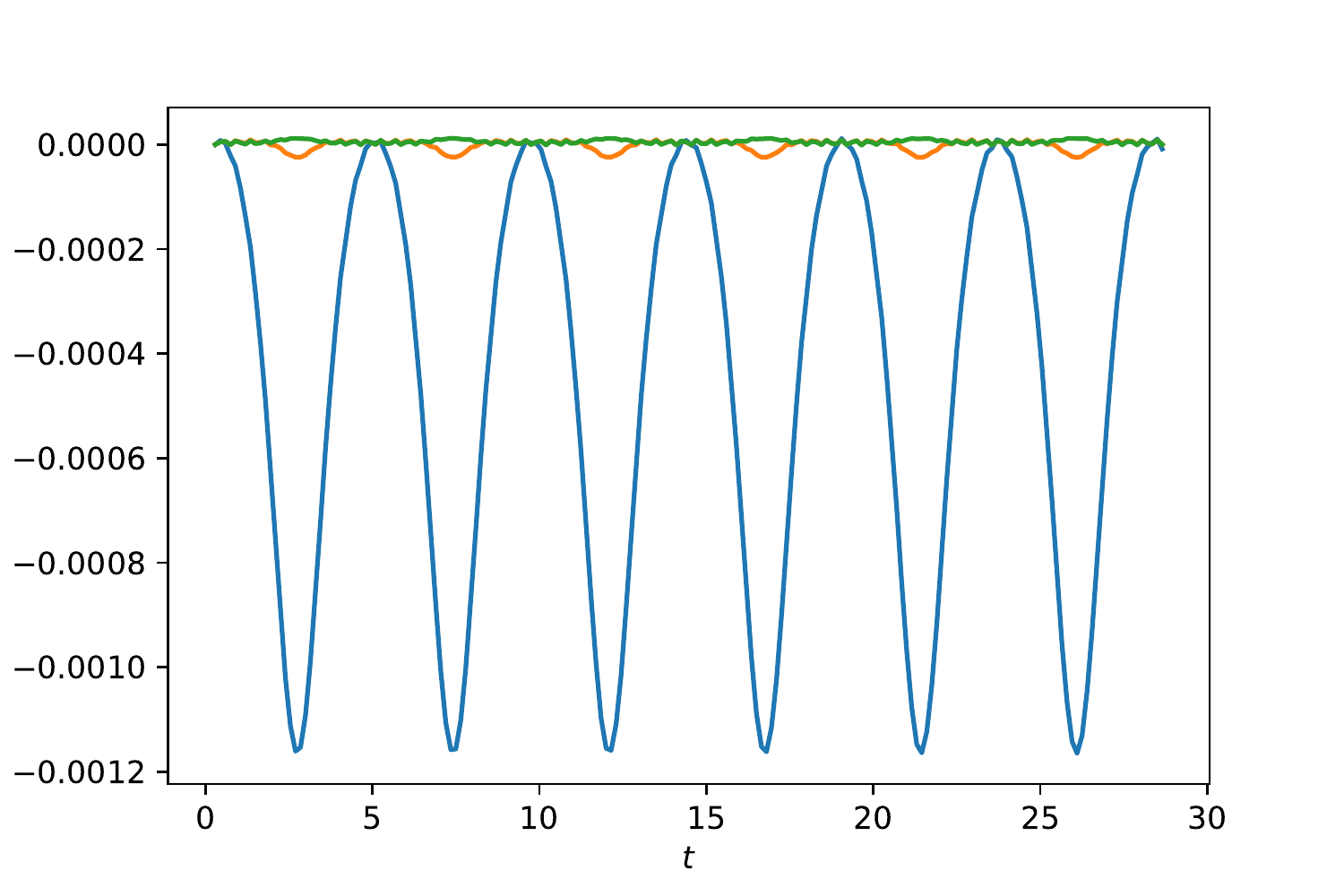}
		
	\end{center}
	\caption{The plots in the second row show the error $I-I(y_{\mathrm{init}},\dot{y}_{\mathrm{init}})$ (blue) in the angular momentum and $I^{[2]}_\mod-I^{[2]}_\mod(y_{\mathrm{init}},\dot{y}_{\mathrm{init}})$ (orange) and $I^{[4]}_\mod-I^{[4]}_\mod(y_{\mathrm{init}},\dot{y}_{\mathrm{init}})$ (green) along a trajectory. The same data as in \cref{fig:MSEx} was used apart from $A_2=\diag(0.3,0.3)$ and $h=0.15$. The plot at the top shows a trajectory initialised as in \cref{fig:MSEx}.}\label{fig:MSExDiag}
\end{figure}

For source code of our experiments refer to \cite{ConjugateSymplecticSoftware}.

\section{Validity of the construction method}\label{sec:proofs}

Let us prepare the proof of \cref{thm:Main}. 
We will use the language of Differential Geometry.

To motivate the following proposition and to fix notation recall the following fact: if $(Z',\omega',H')$ is a Hamiltonian system with Hamiltonian vector field $X'_{H'}$, $Z \xhookrightarrow{\Psi}Z'$ an embedding such that $\Psi(Z)$ is a symplectic submanifold that is invariant under motions, i.e.\ $X'_{H'}(\Psi(z)) \in T_{\Psi(z)} \Psi(Z)$ for all $z \in Z$, then the pull-back system $(Z,\omega,H)$ with $\omega = \Psi^\ast \omega'$, $H = H' \circ \Psi$ is a Hamiltonian system with Hamiltonian vector field $X_H$. Here $\Psi^\ast \omega'$ denotes the pull-back of $\omega'$ along $\Psi$. The Hamiltonian vector fields $X_H$ and $X'_{H'}$ relate by
\[
\Psi_\ast X_H = X'_{H'} \circ \Psi,
\]
i.e.\ for a motion $\gamma$ with $\dot{\gamma} = X_H \circ \gamma$ the curve $\gamma' = \Psi \circ \gamma$ is a motion of $(Z',\omega',H')$, i.e.\ $\dot{\gamma}'=X'_{H'} \circ \gamma'$. Here $\Psi_\ast$ denotes the push-forward map, i.e.\ $(\Psi_\ast X_H)(z) = \d \Psi|_z(X_H(z))$ for $z \in Z$.

In the following, we will adapt this statement to a setting, where the definition of $\omega'$, $H'$, and $\Psi'$ contain a formal variable $h$ and where $\Psi(Z)$ is left invariant only up to higher order terms in $h$.

\begin{definition}[Formal Hamiltonian system]
Let $N \in \N$ be a truncation index, $Z$ be a smooth manifold, $h$ a formal variable, and $H = \sum_{k=0}^N h^k H_k$ a formal polynomial whose coefficients are smooth maps $Z \to \R$. Further, let $\omega=\sum_{k=0}^{ N} h^k \omega_k $ be a formal polynomial whose coefficients are 2-forms on $Z$. 
The collection $(Z,\omega,H)$ is called a {\em formal Hamiltonian system with truncation index $N$} if the following conditions are satisfied.
\begin{itemize}
\item 
The formal symplectic form $\omega$  is closed, i.e.\ $\d \omega = \sum_{k=0}^{ N} h^k \d \omega_k $ is a formal polynomial whose coefficients are 3-forms that are zero.

\item 
The 2-form $\omega_0$ is non-degenerate.
\end{itemize}

The relation
$-\d H = \omega(X, \cdot)$
defines a formal power series $X = \sum_{k=0}^\infty h^k X_k $, where $X_k$ are vector fields on $Z$. The truncation $X_H := \sum_{k=0}^N h^k X_k $ is called {\em (formal) Hamiltonian vector field.}  The formal differential equation $\dot{\gamma} = X_H \circ \gamma$ is called {\em Hamilton's equation}.

\end{definition}

\begin{proposition}\label{prop:PullbackHam}
Let $N \in \N$ be a truncation index, $Z'$ and $Z$ be manifolds, and let $(Z',\omega',H')$ be a formal Hamiltonian system with truncation index $N$. Consider a formal polynomial $\Psi  = \sum_{k=0}^N \Psi_k h^k$ such that $\Psi_k \colon Z \to Z'$ are smooth and $\Psi_0 \colon Z \to Z'$ is an embedding. To $z \in Z$ define formal tangent spaces
\[
T_{\Psi(z)}\Psi(Z) := \left\{ \sum_{k=0}^{N} h^k \d \Psi_k|_z(v)  \Bigg | v \in T_z Z \right\} \subset \bigoplus_{k=0}^N h^k T_{z_k'}Z',
\]
where $z' := \sum_{k=0}^N h^k z_k'  := \Psi(z)$.
\begin{itemize}
\item
Assume that the Hamiltonian vector field is tangential to $\Psi(Z)$, i.e.\ $\trk (X'_{H'}(\Psi(z))) \in T_{\Psi(z)}\Psi(Z)$ for all $z \in Z$. Here
$\trk$ denotes the truncation of $\mathcal{O}(h^{N+1})$-terms.

\item 
Assume that $\Psi(Z)$ is a symplectic submanifold of $Z'$, i.e.\ for all $z \in Z$ 
\[
T_{\Psi(z)} \Psi(Z) \cap T_{\Psi(z)} \Psi(Z)^{\bot_{\omega'}} =0.
\]
Here $\Psi(Z)^{\bot_{\omega'}}$ denotes the symplectic complement
\[
T_{z'} \Psi(Z)^{\bot_{\omega'}}
= \left\{
v \in \bigoplus_{k=0}^N h^k T_{z_k'}Z' \Bigg | \trk(\omega'(v,w)) =0 \; \forall
w \in T_{z'} \Psi(Z)
\right\},
\]
where $z' := \sum_{k=0}^N h^k z_k'  := \Psi(z)$. 

\end{itemize}

Then the pull-back system $(Z,\omega,H)$ with $\omega = \Psi^\ast \omega'$, $H=H' \circ \Psi$ constitutes a formal Hamiltonian system with truncation index $N$ and
\[
\trk(\Psi_\ast X_H) = {X'}_{H'} \circ \Psi.
\]

\end{proposition}

\begin{proof}
	Let $\omega = \sum_{k=0}^N h^k\omega_k $
	Since $\Psi_0$ is an embedding, it is an immersion and we conclude that $\omega_0$ is non-degenerate as $\omega_0'$ is non-degenerate. Since pull-back and the differential $\d $ commute, $\omega$ is closed.

	In the following calculations we identify two formal polynomials $P_1$ and $P_2$ with coefficients of the same type (real numbers, $n$-forms, smooth functions, $\ldots$) if and only if $P_1-P_2 \in \mathcal{O}(h^{N+1})$.
For all $z \in Z$ 
we have
\begin{align*}
	-\omega'_{\Psi(z)}( \d \Psi|_z(X_H(z)), \d \Psi|_z (\cdot))
	&= -(\Psi^\ast\omega')_{z}( X_H(z), \cdot)
	= -\omega_{z}( X_H(z), \cdot)\\
	&= \d H|_z
	=\d (H' \circ \Psi)|_z
	=\d H'|_{\Psi(z)} \circ \d \Psi|_z\\
	&= -\omega'_{\Psi(z)}( X'_{H'}(\Psi(z)), \d \Psi|_z(\cdot))
\end{align*}
\begin{align*}
&\implies
\omega'_{\Psi(z)}( \d \Psi|_z(X_H(z))-X'_{H'}(\Psi(z)), \d \Psi|_z (\cdot)) =0\\
&\implies
 \d \Psi|_z(X_H(z))-X'_{H'}(\Psi(z)) \in T_{\Psi(z)}\Psi(Z)^{\bot_{\omega'}}.
\end{align*}

As the inclusion $\d \Psi|_z(X_H(z))-X'_{H'}(\Psi(z)) \in T_{\Psi(z)}\Psi(Z)$ holds as well and $\Psi(Z)$ is a symplectic submanifold, it follows that 
\[
\d \Psi|_z(X_H(z))=X'_{H'}(\Psi(z)). \qedhere
\]
\end{proof}

We can now proceed to the proof of \cref{thm:Main}. 

\begin{proof}
	{\em Step 1. Construction and validity of the modified equation.}
The truncated power series $\mathcal{L}_\Delta^{[N]}$ is a formal polynomial of the form
\[
\mathcal{L}_\Delta^{[N]}(y^{[M]}) 
= L(y,\dot{y}) + \sum_{k=1}^N h^k \mathcal{L}_\Delta^{k}(y^{[M]}),
\]
where $h$ is the formal variable.
Since the Lagrangian $L$ is regular, the Euler--Lagrange equations to $\mathcal{L}_\Delta^{[N]}$
\[
\sum_{j=0}^M (-1)^j \frac{\d }{\d t} \left( \nabla_{y^{(j)}} \mathcal{L}_\Delta^{[N]}(y^{[M]}) \right)=0
\]
are equivalent to an ordinary differential equation of the form
\begin{equation}\label{eq:ModODEHigh}
\ddot{y} = g_0(y,\dot{y}) + \sum_{k=1}^N h^k \tilde{g}_k (y^{[2M]})
\end{equation}
when truncating terms of order $\mathcal{O}(h^{[N+1]})$.
The ordinary differential equation is of order $2M$ because $\mathcal{L}_\Delta^{[N]}(y^{[M]})$ is regular as well.
Iterative replacements of derivatives of order $j\ge 2$ by derivatives of \eqref{eq:ModODEHigh} yield the modified equation

\begin{equation}\label{eq:ModODE}
\ddot{y} = g_0(y,\dot{y}) + \sum_{k=1}^N h^k g_k (y,\dot{y}) + \mathcal{O}(h^{[N+1]}).
\end{equation}


Solutions to
\begin{equation}\label{eq:ModODETruncated}
	\ddot{y} = g_0(y,\dot{y}) + \sum_{k=1}^N h^k g_k (y^{[M]})
\end{equation} 
fulfil \eqref{eq:ModODEHigh} up to $\mathcal{O}(h^{[N+1]})$ terms by construction. This proves the first part of \cref{thm:Main}.

{\em Step 2. Existence of Hamiltonian structure on a higher jet-space.}
Let $Y$ denote the domain of $y$ (smooth manifold), $Z = \Jet^1(Y)$ the 1-jet space, and $Z' = \Jet^{2M-1}(Y)$.
The iterative substitution procedure gives rise to a formal polynomial $\Psi$ of maps $Z \to Z'$ defined by $y^{[1]} \mapsto y^{[2M-1]}$, where $y^{(j)}$ for $j\ge 2$ is expressed as a formal polynomial of functions depending on $y^{[1]} = (y,\dot{y})$.
The expression for $y^{(j)}$ is obtained by deriving \eqref{eq:ModODEHigh} $j-2$ times, iteratively replacing derivatives of order greater than 2 by derivatives of \eqref{eq:ModODEHigh} followed by a truncation of $\mathcal{O}(h^{[N+1]})$ terms.

In the remainder of this part of the proof we suppress the fact that we operate on formal polynomials as the following steps can also be done when $h$ is substituted with a sufficiently small real number $h>0$.


In the following, we will show that there exists a Hamiltonian structure $(Z',\omega',H')$ on $Z'$ whose motions correspond to the motions induced by the order $M$-Lagrangian $\mathcal{L}_\Delta^{[N]}(y^{[M]})$. To compute the Hamiltonian structure $(Z',\omega',H')$, we first construct a Hamiltonian system on $T^\ast \Jet^{M-1}(Y)$ which we will then pull back to $Z'$.

As the Lagrangian $\mathcal{L}_\Delta^{[N]}(y^{[M]})$ is regular, by Ostrogradsky's principle for high order Lagrangians \cite{whittaker1988}, there exists a transformation $\chi \colon Z' \to T^\ast \Jet^{M-1}(Y)$ between the jet variable $y^{[2M-1]}\in Z'$ and variables $(q,p) \in T^\ast \Jet^{M-1}(Y)$ with
\[ q = (y,\dot y,\ldots, y^{(M-1)})\]
and
\[
p_i
= \sum_{k=0}^{M-i} (-1)^k \frac{\d^{k}}{\d t^{k}}
\nabla_{y^{(k+i)}}\mathcal{L}^{[N]}_\Delta(y^{[M]}), \quad i=1,\ldots,M
\]
such that with 
\[
\Omega=\sum_{i=1}^M \sum_{j=1}^n \d p_i^j \wedge \d q^i_j,
\quad \text{and} \quad
\mathcal{H}^{[N]}(q,p) = \sum_{k=1}^M \langle p_k,q^k \rangle - \mathcal{L}_\Delta^{[N]}
\]
the motions of \eqref{eq:ModODEHigh} are exactly mapped to the motions of the Hamiltonian vector field $X_{\mathcal{H}^{[N]}}$ on $T^\ast \Jet^{M-1}(Y)$ with
\begin{equation}\label{eq:BigOmegaVf}
-\d \mathcal{H}^{[N]} = \Omega(X_{\mathcal{H}^{[N]}}, \cdot)
\end{equation}
by the transformation $\chi^{-1}$. 
Let $\omega' = \chi^\ast \Omega$ and $H' = \mathcal{H}^{[N]} \circ \chi$. Now $(Z',\omega',H')$ is a Hamiltonian system on $Z'$ whose motions are exactly the motions induced by the Lagrangian structure.
This completes step 2 of the proof.

\begin{remark}
In the setting of formal polynomials, $\Omega$ is a two form whose coefficients are formal polynomials or, alternatively, $\Omega$ is a formal polynomial whose coefficients are 2-forms. In this case, \eqref{eq:BigOmegaVf} defines a formal series $X_{\mathcal{H}^{[N]}}$ in $h$ from which we truncate $\mathcal{O}(h^{N+1})$ terms. This is also done when defining ${X'}_{H'}$ through $-\d H' = \omega'({X'}_{H'}, \cdot)$. The differential equation defined by ${X'}_{H'}$ is then equivalent to the differential equation induced by the Lagrangian structure up to $\mathcal{O}(h^{N+1})$ terms. 
\end{remark}

{\em Step 3. Pull-back of $(Z',\omega',H')$ to $Z$.}
As the Lagrangian $L$ is regular, the 2-form $\omega_0'=\sum_{j=1}^n\d y_j \wedge \d \mathfrak{p}^j$, where $\mathfrak{p}$ is obtained by the Legendre transform for $L$, is a symplectic form.
The pull-back form $\omega=\Psi^\ast\omega'$ is a formal polynomial in $h$, whose coefficients are 2-forms. The zeroth coefficient $\omega_0$ coincides with $\omega_0'$. $\omega$ is, therefore, non-degenerate and $\Psi(Z)$ is a symplectic submanifold of $Z'$ in the sense of \cref{prop:PullbackHam}. Moreover, by construction of $H'$ and $\Psi$ the condition $\trk (X'_{H'}(\Psi(z))) \in T_{\Psi(z)}\Psi(Z)$ for all $z \in Z$ is fulfilled. Therefore, \cref{prop:PullbackHam} completes the theorem's proof.
%
%
%
%
%
%
%
%
\end{proof}

\begin{proof}[Proof of \cref{thm:HJmod}]
	The construction method of the modified data $J^{[N]}_\mod$ and $H^{[N]}_\mod$ coincides with the construction method verified in the proof of \cref{thm:Main}.

\end{proof}

\section{On the existence of modified Lagrangians}\label{sec:ExistenceLmod}

\begin{proposition}\label{prop:HamS}
All Hamiltonian systems 
\[
\dot{z}(t) = \bar{J}^{-1}(z(t))\nabla \bar{H}(z(t)), \quad z=\begin{pmatrix}
y\\ \dot y
\end{pmatrix}
\]
with a fixed symplectic structure represented by $\bar{J}$ can be formulated as variational problems of the form
\begin{equation}
\label{eq:VarPrincSbar}
\delta \bar{S} =0
\quad
\text{for}
\quad
\bar{S}(y) = \int \bar{L}(y(t),\dot{y}(t))\d t
\end{equation}
if and only if $\bar{J}$ is of the form
\begin{equation}\label{eq:Jbar}
\bar{J} = \begin{pmatrix}
\ast & \ast\\
\ast & 0_n
\end{pmatrix},
\end{equation}
where $0_n$ denotes an $n\times n$-dimensional zero matrix with $n$ the dimension of $y$, (i.e.\ if the distribution spanned by the vector fields $\frac{\p}{\p \dot{y}_1},\ldots,\frac{\p}{\p \dot{y}_n}$ is Lagrangian).
\end{proposition}

\begin{proof}
Denote the domain of the variable $y$ by $Y$, the 1-jet space over $Y$ by $\mathrm{Jet}^1(Y)$, and the symplectic form represented by the matrix $\bar{J}$ by $\bar{\omega}$. Consider a Hamiltonian $\bar{H}\colon \mathrm{Jet}^1(Y) \to \R$.
If the distribution $\mathcal{D}$ spanned by $\frac{\p}{\p \dot{y}_1},\ldots,\frac{\p}{\p \dot{y}_n}$ is Lagrangian w.r.t.\ $\bar{\omega}$, then there exists a primitive $\bar{\lambda}$ of $\bar{\omega}$ with kernel $\mathcal{D}$ \cite[Cor.15.7]{Libermann1987}. The 1-form $\bar{\lambda}$ is of the form
\[
\bar{\lambda}=\sum_{j=1}^n \ell(y,\dot{y})\, \d y_j.
\]
By Hamilton's principle, a curve $\gamma\colon [t_{\mathrm{start}},t_{\mathrm{end}}] \to Y$ is a Hamiltonian motion if the action functional
\[
\bar{S}(\gamma) = \int_\gamma (\bar{\lambda} - \bar{H} \d t)
\]
is stationary w.r.t.\ variations of $\gamma$ through smooth curves fixing the endpoints. Thanks to the special structure of $\bar{\lambda}$ (absence of $\d \dot{y}_j$-terms), the pullback of $\bar{\lambda} - \bar{H} \d t$ along a curve $\gamma$ has the form
\[
\bar{L}(y(t),\dot{y}(t))\d t
\]
with $y(t)$ describing $\gamma(t)$ in the coordinate $y$. Therefore, in coordinates, the variational principle has the form \eqref{eq:VarPrincSbar}.

On the other hand, a variational principle of the form \eqref{eq:VarPrincSbar} with regular Lagrangian $\bar{L}$ (i.e.\ invertible $\left(\frac{\p^2\bar{L}}{\p \dot{y}_i \p \dot{y}_j }\right)_{i,j}$) can be formulated as a Hamiltonian system with Hamiltonian
\[
\bar{H}(y,\dot y) = \langle \dot{y},p(y,\dot y) \rangle - L(y,\dot y), \quad \text{and with}\quad p(y,\dot{y})= \frac{\p\bar{L}}{\p \dot{y}}(y,\dot y).
\]
The symplectic structure is given as
\[
\bar{\omega} = \d \bar{\lambda}, \quad \text{with} \quad \bar{\lambda} =  \sum_{j=1}^n  p_j(y,\dot y) \d y_j.
\]
As the distribution $\mathcal{D}$ spanned by $\frac{\p}{\p \dot{y}_1},\ldots,\frac{\p}{\p \dot{y}_n}$ is in the kernel of a primitive of $\bar{\omega}$, the distribution is Lagrangian.
\end{proof}

\begin{remark}
The strength of \cref{prop:HamS} lies in the assertion that $\bar{L}$ is a first-order Lagrangian in the original variable $y$, i.e.\ it depends on $(y,\dot{y})$ only.
If $\bar{J}$ is {\em not} of the required form, then, by Darboux's theorem, we can perform a change of variables on $\mathrm{Jet}^1(Y)$ such that $\bar{\omega}$ is the standard symplectic form $\sum \d \bar{p}_i \wedge \d \bar{q}_i$ and $\bar{L}$ is a 1st-order Lagrangian in $q$, i.e.\ depends on $(q,\dot{q})$ but not on higher derivatives in $q$.
However, as $q$ and $p$ each depend on $(y,\dot{y})$, the variables have lost their dynamical meaning. This is because the required change of variables on $\mathrm{Jet}^1(Y)$ is not fibred, i.e.\ the jet-space structure is not preserved.
Expressed in the original variable $y$, the Lagrangian $\bar{L}$ then depends on $(y,\dot{y},\ddot{y})$, i.e.\ describes a higher order variational structure.
\end{remark}

\begin{remark}
The computational example presented in \cref{sec:CompEx} provides an example for which the modified symplectic structure is {\em not} of the form that is required in \cref{prop:HamS}, unless $A_1$ is of the form $A_1=\alpha I$, i.e.\ the method coincides with a multistep method with scalar coefficients. 
Indeed, if the considered method is a classical multistep method, i.e.\ all coefficients are scalar, then $L_\mod(y,\dot{y})$ exists by \cref{thm:Smod}.
\end{remark}

We now proceed to the proof of \cref{thm:Smod}. We exploit that linear multistep methods can be interpreted as 1-step methods on the original phase space \cite[\S XV.2]{GeomIntegration}. Here and below we refer to the theory of linear multistep methods for 2nd order ODEs.

\begin{proof}[Proof of \cref{thm:Smod}]
As proved in \cite[\S 5]{Chartier2006}, for the underlying 1-step method $\phi$ of a symmetric linear multistep method there exists a local diffeomorphism $\psi$ such that $\tilde \phi =\psi \circ \phi \circ \psi^{-1}$ is symplectic with respect to the original symplectic structure $\omega=\sum_j \d p_j \wedge \d q_j$. The conjugacy $\psi$ is given as a $P$-series applied to the original Hamiltonian vector field $X_0$ given by the right hand side of
\begin{align*}
\dot{q}&=p\\
\dot{p}&=\nabla U(q).
\end{align*}
The $P$-series $\psi$ is a formal power series that is in general not convergent. Conjugacy, pull-back and push-forward operations are to be interpreted in a formal sense.
The map $\phi$ is symplectic with respect to the modified symplectic structure $\omega_\mod=\psi^\ast \omega$ and is the time-$h$-flow of a vector field $X$, for which the flow equations correspond to a first order formulation of the modified equation \eqref{eq:modODE}. The $\omega$-symplectic map $\tilde{\phi}$ is the time-$h$-flow of the $\psi$-related vector field $\tilde{X} = \psi_\ast X \circ \psi^{-1}$. By standard results on backward error analysis for symplectic integrators \cite[\S IX]{GeomIntegration}, $\tilde X$ is a Hamiltonian vector field w.r.t.\ the standard symplectic structure $\omega$ for a Hamiltonian $H$ up to any order in the step-size $h$. Pulling back the Hamiltonian structure $(\omega,H)$ using $\psi$, we obtain a modified Hamiltonian system $(\omega_\mod,H_\mod) = (\psi^\ast \omega, H \circ \psi)$ such that Hamilton's equations are equivalent to the modified equation \eqref{eq:modODE}.

Since $\psi$ is a $P$-series in $X_0$, the distribution $\mathcal D$ spanned by the vertical vector fields $\frac{\p}{\p p_1},\ldots,\frac{\p}{\p p_n}$ is Lagrangian w.r.t.\ $\omega_\mod=\psi^\ast \omega$.
By \cref{prop:HamS}, for any order $N$ in the step-size $h$ there exists a first order Lagrangian $L_\mod(y,\dot{y})$ in the original variable such that the variational principle
\[
\delta \left(\int L_\mod(y(t),\dot{y}(t)) \d t\right) =0
\]
recovers the modified equation up to higher order terms in $h$.
\end{proof}

\begin{remark}\label{rem:combinatorialStructure}
The proof of \cref{thm:Smod} also shows that $L_\mod$ in \cref{thm:Smod} has the structure of an $S$-series applied to a $P$-series (see \cite{Chartier2006}). The modified Lagrangian $L_\mod$ can, thus, be computed with an ansatz as well.
The modified data $H_\mod$ and $J_\mod$ can then be computed from $L_\mod$ by a Legendre transformation.
\end{remark}

\section{Future work}\label{sec:future}
Motivated by optimal truncation results for modified equations, it would be interesting to analyse the convergence properties of modified symplectic structures, modified Hamiltonians, and modified Lagrangians. Moreover, in view of \cref{rem:combinatorialStructure}, a systematic description of the combinatorial structure of the modified quantities appears feasible.

\bibliography{resourcesnoUrl}
\bibliographystyle{plainurl}

\appendix

\input{BAE_Leapfrog1D.tex}

\end{document}

%% file: ORCID_logo.tex
\definecolor{orcidlogocol}{HTML}{A6CE39}
\tikzset{
  orcidlogo/.pic={
    \fill[orcidlogocol] svg{M256,128c0,70.7-57.3,128-128,128C57.3,256,0,198.7,0,128C0,57.3,57.3,0,128,0C198.7,0,256,57.3,256,128z};
    \fill[white] svg{M86.3,186.2H70.9V79.1h15.4v48.4V186.2z}
                 svg{M108.9,79.1h41.6c39.6,0,57,28.3,57,53.6c0,27.5-21.5,53.6-56.8,53.6h-41.8V79.1z M124.3,172.4h24.5c34.9,0,42.9-26.5,42.9-39.7c0-21.5-13.7-39.7-43.7-39.7h-23.7V172.4z}
                 svg{M88.7,56.8c0,5.5-4.5,10.1-10.1,10.1c-5.6,0-10.1-4.6-10.1-10.1c0-5.6,4.5-10.1,10.1-10.1C84.2,46.7,88.7,51.3,88.7,56.8z};
  }
}

\newcommand\orcidicon[1]{\href{https://orcid.org/#1}{\mbox{\scalerel*{
\begin{tikzpicture}[yscale=-1,transform shape]
\pic{orcidlogo};
\end{tikzpicture}
}{|}}}}

%% file: CustomDefsPreamble.tex
\declaretheorem[name=Definition,style=definition,numberwithin=section,qed={\hfill $\triangle$}]{definition}

\declaretheorem[name=Remark,style=remark,numberwithin=section,qed={\hfill $\triangle$}]{remark}

\crefname{observation}{observation}{observations}
\Crefname{observation}{Observation}{Observations}

\crefname{conjecture}{conjecture}{conjectures}
\Crefname{conjecture}{Conjecture}{Conjectures}

\crefname{assumption}{assumption}{assumptions}
\Crefname{assumption}{Assumption}{Assumptions}



%% file: CustomDefs.tex
\def\im{\mathrm{Im}\, }
\def\re{\mathrm{Re}\, }
\def\diam{\mathrm{diam}\, }
\def\mod{\mathrm{mod}\, }
\def\Hess{\mathrm{Hess}\, }
\def\rank{\mathrm{rank }\,  }
\def\rg{\mathrm{rg }\,  }
\def\d{\mathrm{d}}
\def\p{\partial }
\def\pr{\mathrm{pr}}
\def\D{\mathrm{D}}
\def\id{\mathrm{id}}
\def\Id{\mathrm{Id}}
\def\Jet{\mathrm{Jet}}
\def\e{\epsilon}
\def\trk{\mathrm{trk}}
\def\C{\mathbb{C}}
\def\Z{\mathbb{Z}}
\def\Q{\mathbb{Q}}
\def\N{\mathbb{N}}
\def\R{\mathbb{R}}
\def\K{\mathbb{K}}
\def\T{\mathbb{T}}
\def\diag{\mathrm{diag}\, }

\newtheorem{theorem}{Theorem}[section]
\newtheorem{corollary}[theorem]{Corollary}
\newtheorem{lemma}[theorem]{Lemma}
\newtheorem{proposition}[theorem]{Proposition}
\newtheorem{prop}[theorem]{Proposition}
\newtheorem{conjecture}[theorem]{Conjecture}
\newtheorem{assumption}{Assumption}[section]

\numberwithin{equation}{section}


\makeatletter
\renewcommand*\env@matrix[1][\arraystretch]{%
  \edef\arraystretch{#1}%
  \hskip -\arraycolsep
  \let\@ifnextchar\new@ifnextchar
  \array{*\c@MaxMatrixCols c}}
\makeatother

\newsavebox{\smlmata}
\savebox{\smlmata}{${\scriptstyle\overline q(0)=\left(\protect\begin{smallmatrix}0.2\\ 0.1\protect\end{smallmatrix}\right) = \overline q(5)}$}
\newsavebox{\smlmatb}
\savebox{\smlmatb}{${\scriptstyle q =\left(\protect\begin{smallmatrix} -1&2\\3&1\protect\end{smallmatrix}\right) \overline q}$}
\newsavebox{\smlmatc}
\savebox{\smlmatc}{${\scriptstyle p(0)=\left(\protect\begin{smallmatrix}0.85\\ 1.5\protect\end{smallmatrix}\right) = p(2 \pi/3)}$}

\newcommand{\dd}{\mathrm{d}}
\DeclareRobustCommand\marksymbol[2]{\tikz[#2,scale=1.2]\pgfuseplotmark{#1};}
\DeclareRobustCommand{\CollLine}{\raisebox{2pt}{\tikz{\draw[-.,red,dash pattern={on 7pt off 2pt on 3pt off 2pt},line width = 1.5pt](0,0) -- (9mm,0);}}}
\DeclareRobustCommand{\RefLine}{\raisebox{2pt}{\tikz{\draw[-,black,dash pattern={on 7pt off 2pt},line width = 1.5pt](0,0) -- (9mm,0);}}}
\DeclareRobustCommand{\ConvLine}{\raisebox{2pt}{\tikz{\draw[-,blue,line width = 1.5pt](0,0) -- (9mm,0);}}}
\DeclareRobustCommand{\redtriangle}{\raisebox{0.5pt}{\tikz{\node[draw,scale=0.3,regular polygon, regular polygon sides=3,fill=red!10!white,rotate=0](){};}}}

%% file: BAE_Leapfrog1D.tex
\section{Illustration of blended backward error analysis on a classical example}\label{sec:IlluBEA}

For comparison of blended backward error analysis with classical backward error analysis \cite[\S IX]{GeomIntegration} as well as with Vermeeren's approach \cite{Vermeeren2017}, let us illustrate our method of backward error analysis on a traditional example. 
The 1-dimensional mechanical ODE $\ddot y +\nabla W(y)=0$ arises as the Euler-Lagrange equation to the Lagrangian $L (y,\dot{y}) = \frac 12 {\dot{y}}^2 - W(y)$. The St\"ormer--Verlet scheme corresponds to the discrete Euler--Lagrange equations with discrete Lagrangian \[L_\Delta(y_i,y_{i+1}) = \frac 12 \frac{(y_i-y_{i+1})^2}{h^2} - \frac {W(y_i)+W(y_{i+1})}2.\]
In the above expression $(y_i)_{i \in \Z}$ is a discrete variable which approximates a continuous variable on $\R$ at all points of a uniform grid with spacing $h$. In the following $y$ denotes a continuous variable $y \colon \R \to \R$. 


Following the backward error analysis approach of the paper, we compute a series expansion $\mathcal L^{[4]}_\Delta$ of $L_\Delta(y(t),y(t + h))$ around $h=0$, form Ostrogradsky's Hamiltonian description of high-order Lagrangians, and substitute higher order derivatives of $y$ in the Hamiltonian using the Euler--Lagrange equations to $\mathcal L^{[4]}_\Delta$. We obtain the modified Hamiltonian
\begin{align*}
H^{[4]}_\mod(y,\dot{y})
&=W+\frac{\dot{y}^2}{2}
+\frac{1}{24} h^2 \big(-2 W'' \dot{y}^2-\left(W'\right)^2\big)\\
&+\frac{1}{720} h^4 \big(-2 \left(W''\right)^2 \dot{y}^2+3 W^{(4)} \dot{y}^4-6 W^{(3)} W' \dot{y}^2-3 \left(W'\right)^2 W''\big),
\end{align*}
where $W$ and its derivatives $W',W'',W^{(3)},W^{(4)}$ are evaluated at $y$.
A potential drawback compared to classical backward error analysis is that $H^{[4]}_\mod$ does not correspond to the original symplectic structure $\d y \wedge \d \dot{y}$ which we would obtain via Lagrange transformation for the exact Lagrangian $L$. Instead, we obtain a perturbed symplectic structure $\omega^{[4]}_\mod$ which in the frame $\frac{\p }{\p y},\frac{\p }{\p \dot y}$ is represented by the matrix
\[
J^{[4]}_\mod = \begin{pmatrix}
0&-\omega_{21}\\
\omega_{21}&0
\end{pmatrix}
\]
with
\[
\omega_{21} = 1-\frac{1}{6}h^2 W'' + \frac{1}{180} h^4 \left(3 W^{(4)} (\dot{y})^2-3 W^{(3)} W'-(W'')^2\right).
\]
However, the flexibility in the symplectic structure in our approach allows the computation of modified Hamiltonian structures in cases where the flow is only {\em conjugate} symplectic as in the multipoint Lagrangians considered in this paper. In this example, however, a change of variables is not necessary since the distribution $\mathcal{D}$ spanned by $\frac{\p}{\p \dot{y}}$ is Lagrangian for $\omega^{[4]}_\mod$. Therefore, we find a primitive $\lambda^{[4]}$ of $\omega^{[4]}_\mod$ with kernel $\mathcal{D}$. The primitive is given as
\[
\lambda^{[4]} =  -\left(\int \omega_{21} \d {\dot{y}}\right)\d y.
\]
A modified Lagrangian $L^{[4]}_\mod$ can be obtained from
\[
L^{[4]}_\mod \d t = \lambda^{[4]} - H^{[4]}_\mod\d t
\]
as
\begin{align*}
L^{[4]}_\mod(y,\dot{y})
&=-\left(\int \omega_{21} \d {\dot{y}}\right)\dot{y} - H^{[4]}_\mod \\
&=\frac{1}{2} \big(\dot{y}^2-2 W\big)
+\frac{1}{24} h^2 \left(\left(W'\right)^2-2 W'' \dot{y}^2\right)\\
&+\frac{1}{720} h^4 \big(-2 \left(W''\right)^2 \dot{y}^2+W^{(4)} \dot{y}^4-6 W^{(3)} W' \dot{y}^2+3 \left(W'\right)^2 W''\big).
\end{align*}
